\documentclass[11pt,letterpaper]{amsart}

\pdfoutput=1
\usepackage{amsmath, amsthm, amssymb,slashed,stmaryrd}
\usepackage{comment}
\usepackage{ifpdf}
\usepackage[pdftex]{graphicx}
\usepackage{tikz}
\usetikzlibrary{matrix,arrows,calc}
\usepackage{mathtools}
\usepackage[pdftex,plainpages=false,hypertexnames=false,pdfpagelabels]{hyperref}
\usepackage{cleveref}
\usepackage{csquotes}
\usepackage{amssymb}
\usepackage{tikz-cd}
\usetikzlibrary{arrows,decorations.markings}
\usetikzlibrary{shapes.geometric}

\usepackage{bm}
\usepackage{accents}
\usepackage[super]{nth}
\usepackage{mathrsfs}
\numberwithin{equation}{subsection}

\newtheorem{thm}[subsection]{Theorem}
\newtheorem{lem}[subsection]{Lemma}
\newtheorem{prop}[subsection]{Proposition}
\newtheorem{cor}[subsection]{Corollary}
\newtheorem{notation}[subsection]{Notation}
\newtheorem{setup}[subsection]{Set-up}
\theoremstyle{definition}

\newtheorem{defn}[subsection]{Definition}
\newtheorem{defnprop}[subsection]{Definition/Proposition}
\newtheorem{rmk}[subsection]{Remark}

\newtheorem{claim}[subsection]{Claim}

\newtheorem*{conjecture*}{Conjecture}
\newtheorem*{theorem*}{Theorem}
\newtheorem*{claim*}{Claim}
\newtheorem*{corollary*}{Corollary}
\newtheorem*{notation*}{Notation}

\DeclareSymbolFont{bbold}{U}{bbold}{m}{n}
\DeclareSymbolFontAlphabet{\mathbbold}{bbold}

\def\Hom{{\sf Hom}}

\def\Spec{{\rm Spec}}

\def\Lie{{\rm Lie}}

\def\Gal{{\rm Gal}}

\def\End{{\rm End}}

\def\qbar{{\overline{\mathbb{Q}}}}

\def\K3{{\rm K3}}

\def\Fil{{\rm Fil}}
\def\Imag{{\rm Im}}
\def\GL{{\rm GL}}

\def\Res{{\rm Res}}

\def\Homo{{\rm Hom}}

\def\Qlbar{{\overline{\mb{Q}}_{\ell}}}
\def\Qpbar{{\overline{\mb{Q}}_{p}}}
\def\Frob{{\rm Frob}}
\def\tr{{\rm tr}}
\def\Fpbar{{\overline{\mb{F}}_p}}
\def\pdiv{{\mscr{G}}}

\def\sp{{\rm sp}}

\newcommand{\fisoc}[1]{\textbf{F-Isoc}(#1)}

\newcommand{\defeq}{\vcentcolon=}

\newcommand{\Cbar}{\overline{C}}
\newcommand{\Cprimebar}{\overline{C'}}
\newcommand{\Qlpbar}{\overline{\mathbb{Q}}_{\ell'}}

\newcommand{\pioneprime}{\pi_1^{(p')}}

\makeatletter
\newcommand{\colim@}[2]{%
  \vtop{\m@th\ialign{##\cr
    \hfil$#1\operator@font colim$\hfil\cr
    \noalign{\nointerlineskip\kern1.5\ex@}#2\cr
    \noalign{\nointerlineskip\kern-\ex@}\cr}}%
}
\newcommand{\colim}{%
  \mathop{\mathpalette\colim@{}}\nmlimits@
}
\makeatother

\newcommand\nc{\newcommand}
\nc\mf\mathfrak
\nc\mc\mathcal
\nc\mb\mathbb
\nc\msf\mathsf
\nc\mscr\mathscr

\usepackage[
backend=biber,
style=alphabetic,
sorting=nyt,
maxnames=50
]{biblatex}

\begin{document}

\title{Boundedness of trace fields of rank two local systems}

\author{Yeuk Hay Joshua Lam}
\email{joshua.lam@hu-berlin.de}
\address{Humboldt Universität Berlin,
       Institut für Mathematik- Alg.Geo.,
        Rudower Chaussee 25
        Berlin, Germany}

\date{\today}

\begin{abstract}  
Let $p$ be a fixed prime number, and $q$ a power of $p$. For any curve over $\mb{F}_q$ and any  local system on it, we have a number field generated by the traces of Frobenii at closed points, known as the trace field. We show that as we range over all pointed curves of type $(g,n)$ in characteristic $p$ and rank two local systems with infinite monodromy at infinity, the set of trace fields which are unramified at $p$ and of bounded degree is finite. This proves observations of Kontsevich obtained via  numerical computations, which are in turn closely related to the analogue of Maeda's conjecture over function fields. We also prove a similar finiteness result across all primes $p$. The key ingredients  of the proofs are Chin's theorem on independence of $\ell$ of monodromy groups, and  the boundedness of abelian schemes of $\mathrm{GL}_2$-type over curves in positive characteristics, obtained using partial Hasse invariants; the latter is an analogue of Faltings' Arakelov theorem for abelian varieties in our setting. 
\end{abstract}

\maketitle 
\setcounter{tocdepth}{1}
\tableofcontents

\section{Introduction}
\iffalse 

\begin{thm}
For any $N$ and $g$, there exists $C=C(N,g)$ such that the number of rank two contribution to motivic $\GL(N, \mb{Q})$-local systems on a genus $g$ curve is less than $C$. Moreover, there exists a number field such that all such motivic rank two local systems have image in $\GL(2, K)$.
\end{thm}
\fi 
Let $p$ be a prime number, and $q$ some power of $p$. Let $\bar{C}/\mb{F}_q$ be a smooth, projective curve of  genus $g$, and $Z$ a non-empty subset  of $n$ points of $\bar{C}$; we write $C\defeq \bar{C}-Z$ for the open curve. We will refer to $(\bar{C}, Z)$ as a pointed curve of type $(g,n)$, and to the points in $Z$  as the \emph{cusps} of $C$.

\begin{defn}
For $(\bar{C}, Z)/\mb{F}_q$ as above, let   $\mscr{L}(\bar{C}, Z)$  be  the set of isomorphism classes of  rank two local systems $\mb{L}$ on $C$ with infinite  monodromy around each point in $Z$, and such that $\det \mb{L}\cong \overline{\mb{Q}}_{\ell}(-1)$.
\end{defn}

For such a local system $\mb{L}$, the work of Drinfeld, which was later generalized to arbitrary rank by Lafforgue, implies that there is a unique number field $F$ and an embedding $\sigma: F\xhookrightarrow{} \overline{\mb{Q}}_{\ell}$ such that $\sigma(F)$ is the field generated by the traces of Frobenii at closed points of $C$. We refer to $F$ as the \emph{trace field} of $\mb{L}$, and denote by $\mf{F}(\mscr{L}(\bar{C},Z))$ the set of trace fields of local systems in $\mscr{L}(\bar{C}, Z)$. It is natural to wonder about the distribution of such trace fields; our first main result is the following boundedness statement. 

\begin{thm}\label{thm:main}
Fix a pair $(g,n)$.  Let $\mf{F}_{g,n}\defeq \bigcup_{\bar{C}, Z, q} \mf{F}(\mscr{L}(\bar{C}, Z))$ be the set of trace fields of   local systems in $\mscr{L}(\bar{C}, Z)$,  as $(\bar{C}, Z)$ and $q$ vary over all pointed curves of type $(g,n)$ and  powers of $p$, respectively. Then, for any integer $d$, there are only finitely many fields in $\mf{F}_{g,n}$ with degree $\leq d$ and which are  unramified  at $p$. 
\end{thm}
\begin{rmk}

\iffalse We expect the result to hold even if we  relax the condition that $F$ is split to $F$ simply being unramified at $p$. The splitness condition is used to bound the degrees of Hodge bundles of certain families of abelian varieties; removing it would  involve generalizing the Frobenius untwisting results of Xia's to $p$-divisible groups with actions. This would also allow one  to obtain similar  results in the \emph{horizontal} direction, namely finiteness across all primes.\fi 
 It is also straightforward to see that, if we fix the trace field $F$ and the pointed curve $(\bar{C}, Z)/\mb{F}_q$ as in Theorem~\ref{thm:main}, then only finitely many elements of $\bigcup_{q'} \mscr{L}(\bar{C}_{\mb{F}_{q'}}, Z_{\mb{F}_{q'}})$ have trace field $F$, where $(\bar{C}_{\mb{F}_{q'}}, Z_{\mb{F}_{q'}})$ denotes the basechange  to $\Spec(\mb{F}_{q'})$. On the other hand, it is possible to have a positive dimensional family of curves, all of which admit rank two local systems with the same trace field; this can be ruled  out in certain special cases, as is done in \cite{memaeda}. 

\end{rmk}

%The restriction to unipotent monodromy in Theorem~\ref{thm:main} is not a crucial one, in the sense that we may also allow bounded wild  ramification. We now state such a result. Let $P\subseteq \Gal_{\mb{F}_q((t))}$ denote the wild inertia subgroup.
\iffalse 
\begin{cor}\label{cor:wild}
Let $P \twoheadrightarrow H$ be a finite quotient of wild inertia. Let $\mscr{L}_H(\bar{C}, Z)$ denote the set of rank two local systems on $C$ such that for each $z\in Z$,
\begin{itemize} 
\item the local Galois representation $\rho_z: \Gal_{\mb{F}_q((t))}\rightarrow \GL_2(\overline{\mb{Q}}_{\ell})$, when restricted to $P$, factors through $H$, 
\item and moreover $\rho_z$ does not have potential good reduction. 
\end{itemize} 
Let $\mf{F}_H$ be the set of trace fields in $\mscr{L}_H(\bar{C},Z)$ as $(\bar{C}, Z)$ and $q$ vary over all pointed curves of type $(g,n)$ and  powers of $p$, respectively. Then there are only finitely many fields in $\mf{F}_H$ of degree $\leq d$ and unramified at $p$.
\end{cor}
\begin{rmk}
It would be interesting to know whether one can remove the bounded ramification condition altogether, i.e. whether there exist local systems with increasingly wild ramification and the same trace fields.
\end{rmk}
\fi 
We now state a result where the characteristic $p$ is allowed to vary. For this purpose, for a pointed curve $(\bar{C}, Z)/\mb{F}_q$ in characteristic $p$, we write $\mscr{L}_p(\bar{C}, Z)$ for $\mscr{L}(\bar{C}, Z)$ to emphasize the characteristic of the base field.
\begin{thm}\label{thm:diffp}
There are only finitely many fields in $\bigcup_{p, \bar{C}, Z}\mf{F}(\mscr{L}_p(\bar{C}, Z))$ of degree $d$ and completely split at $p$.
\end{thm}
The heuristic in the next section leads us to expect that \Cref{thm:main} should hold without the unramified-at-$p$ assumption, and Theorem~\ref{thm:diffp} should hold without the splitness condition, though we are not able to prove these: indeed, our method to bound the degree of the Hodge bundle uses partial Hasse invariants, and the obtained bounds depend on the  prime $p$ if the field is not totally split at $p$. In another direction, one can try to remove the infinite monodromy condition in the definition of $\mscr{L}(\bar{C}, Z)$: the main obstruction  here is that such local systems are not known to come from abelian varieties, or some other family of varieties with a \enquote{uniform} description: instead, they are known to arise in the cohomology of moduli of shtukas, whose geometry depends heavily on $\bar{C}$ and $q$. 

\subsection{Context} Our work is motivated by Maeda's conjecture in the function field setting. For example, for each curve $(\bar{C}, Z)/\mb{F}_q$ of type  $(g,n)=(0,4)$, computations of Kontsevich \cite[\S~0.1]{maxim} shows that, in almost all cases,  there are four trace fields, each of degree roughly $(q+1)/4$\footnote{the splitting into four fields comes from the Atkin-Lehner operators}. This contributes to the Maeda philosophy that, generically, trace fields should be as large as possible.  We refer the reader to \cite{memaeda} for more on the analogue of Maeda's conjecture over function fields, as well as references for the number field case.

We were also heavily  influenced  by the results of Faltings \cite{faltings} on Arakelov's theorem for abelian varieties, as well as Deligne's finiteness theorem \cite{delignefinitude}. More precisely, we were motivated by the possibility of uniformity in Deligne's finiteness theorem: the latter says that on a fixed complex curve $C$, only finitely many rank $N$ $\mb{Q}$-local systems can come from algebraic geometry, and by uniformity we mean whether this finite number depends on the underlying curve in its moduli space. As a first step towards this question, for a fixed $N$, one may ask about the contribution to rank $N$ $\mb{Q}$-local systems coming from abelian varieties: in this case  uniformity is known and  is a corollary of  Faltings' theorem \cite[Theorem~1]{faltings}. Note that this uniformity from Faltings' proof seems, at least to the author, to be stronger than the subsequent re-proofs of the same result due to Deligne \cite{delignefinitude}, as well as Jost and Yau \cite{jostyau}, although these works are more general in that they  apply to local systems not coming from abelian schemes. 

In any case, for fixed $(g,n)$, one sees that there are only finitely many $F$'s of fixed degree such that a curve of type $(g,n)$ carries a non-trivial family of abelian varieties of $\GL_2(F)$-type. Theorem~\ref{thm:main} is the analogue of this in positive characteristic. It would be interesting to investigate the analogue of the full strength version (i.e. beyond abelian varieties of $\GL_2$-type) of Faltings' theorem in positive characteristic. For example, is it true that, in characteristic $p$,  any abelian scheme of dimension $D$ on a curve $(\bar{C}, Z)$ of type $g,n$ is isogenous to one whose Hodge bundle has degree bounded by $g,n,p,D$? We should also mention a related result of Litt \cite{litt} which is an analogue of Deligne's result in positive characteristic, but for $\ell$-adic coefficients.

\subsection{Sketch of proof of Theorem~\ref{thm:main}}
For this sketch, we will focus on the case of local systems with unipotent monodromy around each cusp, which is the key part. By work of Drinfeld, we know that the rank two local $\Qlbar$-local systems in question come from abelian schemes of $\GL_2(F)$-type, with $F$ being the trace field of Frobenii; suppose that there are infinitely many such $F$'s of degree $d$. Using Zarhin's trick, we obtain prinipally polarized abelian schemes of dimension $N\defeq 8d$ over $C$, and therefore infinitely many maps $C\rightarrow \mscr{A}_N$, which extend to maps $\bar{C}\rightarrow \mscr{A}_N^*$, where the latter denotes the minimal compactification. The first key idea  is that, using partial Hasse invariants and  a  Frobenius untwisting result,  we can pass to an isogenous abelian scheme and bound the degree (with respect to the Hodge bundle of $\mscr{A}_N^*$) of such maps in terms of just $(g,n, p,d)$.  This step may be seen as an enhancement of the recent beautiful work \cite{zuoconstruct} of Krishnamoorthy, Yang, and Zuo by the use of partial Hasse invariants; it is also where the unramified at $p$ condition is used.

Now consider the moduli space $\mc{M}$ of curves $C$ of type $(g,n)$, along with maps $\bar{C}\rightarrow \mscr{A}_N^*$ of some fixed  degree. Using the previous step, we have infinitely many points $s_i\in \mc{M}$, corresponding to  infinitely many fields $F_i$'s. Since $\mc{M}$ is of finite type, taking the Zariski closure of the $s_i$'s give a positive dimensional family of curves $\mc{C}\rightarrow S$, and an abelian scheme $\mc{A}\rightarrow \mc{C}$, such that the fiber of $\mc{A}$ at $s_i\in S$ is a (power of an) abelian scheme of $\GL_2(F_i)$-type. If we were in characteristic zero, this would already give a contradiction  since by isomonodromy all the monodromy reprsentations of $\pi_1(\mc{C}_{s_i})$ must be the same. In our situation, the $\mb{Q}$-structure on Betti cohomology is of course not available, and  we crucially make  use of   Chin's theorem on $\ell$-independence of monodromy groups and some tricks, such as the finiteness of number fields of fixed degree and bounded ramification,  to conclude.

\subsection*{Acknowledgements} I am grateful to Maxim Kontsevich for sharing his computations and insights with me, as well as  to Mark Kisin, Bruno Klingler and Sasha Petrov for several enlightening discussions and comments on a previous draft. As is hopefully clear throughout the text, we owe a great  intellectual debt to  the authors of \cite{zuoconstruct}, as well as  the previous work \cite{snowdentsimerman}; we thank both sets of authors for their beautiful work. 

\subsection{Notation} Throughout, $p>0$ will denote a prime, $q$ a power of $p$. For a scheme $Z$, $\pi_1(Z)$ will always denote the \'etale fundamental group. 

\section{Drinfeld's work on function field Langlands}
We first recall the results of Drinfeld from his work on function field Langlands for $\GL_2$. Throughout this section, we use notation as follows. As in the introduction, let $\bar{C}/\mb{F}_q$ be a smooth, projective curve of  genus $g$, and $Z$ an effective Cartier divisor on  $\bar{C}/\mb{F}_q$ of degree $n$; we write $C\defeq \bar{C}-Z$ for the open curve. We will refer to $(\bar{C}, Z)$ as a pointed curve of type $(g,n)$.

Let $\ell$ be a prime distinct from $p$, and $\mb{L}$ be a rank two $\overline{\mb{Q}}_{\ell}$-local system on $C/\mb{F}_q$,  such that $\det \mb{L}\cong \overline{\mb{Q}}_{\ell}(-1)$. Let $F\subset \Qlbar$ be the  field generated by Frobenius traces of  $\mb{L}$; we write  $[F:\mb{Q}]=d.$, and sometimes refer to $F$ simply as the Frobenius trace field, or simply trace field,  of $\mb{L}$.
\begin{thm}\label{thm:drinf}
Suppose $\mb{L}$ has infinite local monodromy around some $z\in Z$. Then there exists an abelian scheme $\pi_{Drinf}: B_{Drinf}\rightarrow C$ of relative dimension $d$, such that $\End_C(B_{Drinf})\otimes \mb{Q}=F$, and $\mb{L}$ appears as a direct summand of $R^1\pi_{Drinf, *}\overline{\mb{Q}}_{\ell}$. Moreover, if we write  $\mb{D}\defeq R^1\pi_{Drinf, *}\Qpbar\in \fisoc{C}_{\Qpbar}$, then we have a decomposition 
\[
\mb{D}=\bigoplus_{\tau} \mb{D}_{\tau},
\]
where the above sum is over embeddings $\tau: F\xhookrightarrow{} \Qpbar$,  each $\mb{D}_{\tau}$ is two dimensional, and the induced action of $F$ on $\mb{D}_{\tau}$ is through $\tau$.

\end{thm}
We refer the reader to \cite[Theorem~2.2]{zuoconstruct} as well as to the remark in loc. cit. that follows for how to deduce the above theorem from the works of Drinfeld. Our next goal is  to refine the abelian scheme $B_{Drinf}$ as follows, under further assumptions.

\begin{defn}\label{defn: av-gl2-type}
    For any scheme $Z/\mb{F}_q$ and number field $F$, we say that an abelian scheme $\pi: B\rightarrow Z$ is of $\GL_2(F)$-type if it satisfies the properties in \Cref{thm:drinf}: that is 
    \begin{itemize}
    \item $\End(B)\otimes \mb{Q}=F$, 
        \item for $\ell\neq p$,  there is a decomposition 
        \[
        R^1\pi_*\Qlbar=\bigoplus_{\tau} \mb{L}_{\tau},\]
        where each $\mb{L}_{\tau}$ has rank two, the induced action of $F$ on $\mb{L}_{\tau}$ is via the embedding $\tau$, and $\det \mb{L}_{\tau}\cong \Qlbar(-1)$. Note that the field of Frobenius traces of each $\mb{L}_{\tau}$ is $F$. 
    \end{itemize}
\end{defn}
We recall the definition of companions.
\begin{defn}
    Let $Y/\mb{F}_q$ be a smooth scheme, and $L$  an $\Qlbar$-local system on $Y$. Let $\ell'\neq p$ be a prime number (possibly equal to $\ell$). Let $\iota \colon \Qlbar\rightarrow \Qlpbar$ be a (possibly non-continuous) field isomorphism; abusing notation, let $\iota$ also denote the induced isomorphism $\Qlbar[t]\rightarrow \Qlpbar[t]$. A \emph{$\iota$-companion to $L$} is a lisse $\Qlpbar$-Weil sheaf $L'$ on $Y$, such that for all closed points $y$, we have:
    $$\iota(P_y(L,t))=P_{y}(L',t)\in \Qlpbar[t].$$
\end{defn}

\begin{prop}
    For an abelian scheme $\pi: B\rightarrow Z$ of $\GL_2(F)$-type, and let $\mb{L}_{\tau}$'s be as in \Cref{defn: av-gl2-type}. For any $\tau, \tau'$, the  local systems $\mb{L}_{\tau}, \mb{L}_{\tau'}$ are companions. Moreover, every $\Qlbar$-companion of $\mb{L}_{\tau}$ is isomorphic to $\mb{L}_{\tau'}$ for some $\tau'$.
\end{prop}
\begin{proof}
    This is recorded in \cite[Remark 2.8]{krishnapal}.
\end{proof}
\begin{comment}
\begin{proof}
    By definition (and Cebotarev), for a fixed $\ell\neq p$, the number of isomorphism classes of $\Qlbar$-companions of $\mb{L}_{\tau}$ is precisely $[F:\mb{Q}]$, where $F$ is the Frobenius trace field of $\mb{L}_{\tau}$. 

    By works of Lafforgue and Drinfeld (see \cite[Theorem 1.1]{drinfeld2010conjecture}), all $\iota$-companions exist for $\iota: \Qlbar\simeq \Qlbar$. By counting, it therefore suffices to show that every $\iota$-companion of $\mb{L}_{\tau}$ is a subobject of  $R^1\pi_*\Qlbar$. 
\end{proof}
\end{comment}

Given a local system $\mb{L}$ on $C$, and $z\in Z$, we can restrict it to the punctured neighborhood of $z\in \bar{C}$: upon picking a local coordinate $t$ at $z$, we obtain  a representation of $\Gal(\mb{F}_q((t)))$. We say that $\mb{L}$ has unipotent local monodromy at $z$ if the inertia subgroup  $I\subset \Gal(\mb{F}_q((t)))$ acts unipotently on this representation. We say that the monodromy is infinite at $z$ if the representation of $I$ does not have finite image.
\begin{prop}\label{prop:drinfeld}
Suppose that $\mb{L}$ has infinite unipotent monodromy around each $z\in Z$, and that its trace field $F$ is unramified at $p$.  Then there  exists an abelian scheme $\pi: B\rightarrow C$ of relative dimension $d$,  with N\'eron model $\bar{B}\rightarrow \bar{C} $ such that 
\begin{enumerate}
\item $B$ has semi-stable and totally degenerate reduction around each $z\in Z$,
\item $\End_C(B)\otimes \mb{Q}=F$, 
    \item $\mb{L}$ appears as a direct summand of $R^1\pi_*\overline{\mb{Q}}_{\ell}$,
    \item the abelian scheme $(B\times B^t)^4$ admits a principal polarization, inducing a  map $C\rightarrow \mscr{A}_{8h}$. The latter extends to a map
    \[
    \bar{f}: \bar{C}\rightarrow \mscr{A}_{8d}^*,
    \]
     such that the line bundle $\bar{f}^*\omega$ has  degree bounded above by a function of $g,n,d,$ and $p$ (crucially,  the dependence is on $p$ and not $q$). Here, $
     \mscr{A}_N^*$ denotes the minimal compactification of the moduli stack of principally polarized abelian varieties of dimension $N$, with  $\omega$  the  Hodge line bundle on it. 
\end{enumerate}
\end{prop}
The proof of this will appear at the end of \S~\ref{section:untwist}.
\begin{rmk}\label{rmk:split}
In the case where $F$ is completely split above $p$, the above  follows from \cite[Lemma~2.7]{zuoconstruct}; in fact, in this case the degree of $\bar{f}^*\omega$ is boundedly above by  a function of only $g,n,d$. 
\end{rmk}
\iffalse 
\begin{proof}
This is contained in \cite[Lemma 2.7]{zuoconstruct}.
\end{proof}
\fi

\section{Partial Hasse invariants}
For any  abelian scheme  $A$  over $C/\mb{F}_q$ with semi-stable reduction along $Z$, we let $\bar{A}\rightarrow \bar{C}$ denote its N\'eron model. Let $\omega_A$  denote its  Hodge vector bundle, which is a vector bundle of dimension $
\dim A$ on $C$ ; similarly we denote by $\omega_{\bar{A}}$ the  Hodge vector bundle of $\bar{A}$, which is a vector bundle over $\bar{C}$, whose restriction to $C$ is given by $\omega_A$.

\begin{setup}\label{setup:avB} For the rest of this section we suppose $B\rightarrow C$ is an abelian scheme with  semi-stable and totally degenerate reduction along $Z$, and moreover $\End(B)\otimes \mb{Q}=F$, which is unramified at $p$ and such that   $\dim B=[F:\mb{Q}]=d$. It is straightforward to check that, passing to an isogenous abelian scheme if necessary, we can (and do)  assume that the ring of integers $\mc{O}_F\subset F$ acts on $B$; we refer to such abelian schemes as $\GL_2(\mc{O}_F)$-type.
%By the Tate conjecture for endormophisms of abelian varieties over function fields, proven by Zarhin for $\ell\neq p$ and de Jong \cite{dejongihes} for $\ell=p$, by passing to an abelian variety isogenous to $B$ over $C$ if necessary, we may assume the ring of integers  $\mc{O}_F\subset F $ acts on $B$. 
\end{setup} 
%let $\bar{B}\rightarrow \bar{C}$ denote its N\'eron model, which we may view as a logarithmic abelian variety on $(\bar{C}, Z)$ equipped with its natual log structure. 

\begin{notation}   Suppose  $
\mb{F}_q$ contains  all the residue fields $k_{\mf{p}}$ for primes $\mf{p}$ of $\mc{O}_F$ lying above $p$; for each $\tau:k_{\mf{p}}\xhookrightarrow{} \mb{F}_q$,  let $\omega_{B, \tau}$ denote the summand of $\omega_B$ on which the $\mc{O}_F$-action   is through $\mc{O}_F\rightarrow k_{\mf{p}}\xhookrightarrow{\tau} \mb{F}_q$. We define $\omega_{\bar{B}, \tau}$ similarly. When the context is clear, we sometimes omit the subscripts $B, \bar{B}$ and simply denote the Hodge bundles by $\omega_{\tau}$. Similarly, whenever there is a module with an  $\mc{O}_F$-action, we use the subscript $\tau$ to denote the component where the action is through $\tau$. 
\end{notation}

\begin{defnprop}
\begin{enumerate} 
\item 
Let $(M,F,V)$ denote the (covariant) logarithmic Dieudonn\'e crystal on $(\bar{C}, Z)$ associated to $\bar{B}$, as constructed by Kato and Trihan \cite[\S~4]{katotrihan}\footnote{see also \cite[Appendix A]{krishnapal} for a very nice summary}. We  denote by $M_{(\bar{C}, Z)}$ the evaluation of the crystal on the trivial thickening $(\bar{C}, Z)$; this is a vector bundle on $\bar{C}$ with integrable connection  with  logarithmic poles along $Z$, and  is moreover equipped with Frobenius and Verschiebung maps $F$ and $V$. 
\item For each prime $\mf{p}$ of $F$ and each embedding $\tau: k_{\mf{p}}\xhookrightarrow{} \mb{F}_q$, let $H_{1, \tau}^{dR}\defeq M_{(\bar{C},Z), \tau}$ \footnote{the notation here is to reflect that, when restricted to $C$, the bundle is given by the $\tau$-component of the relative de Rham \emph{homology} of $B$} denote the summand on which the $\mc{O}_F$-action is through $\mc{O}_F\rightarrow k_{\mf{p}}\xhookrightarrow{\tau} \mb{F}_q$.  The latter is  equipped with a Hodge filtration
\[
0\rightarrow \omega_{\bar{B}^t, \tau} \rightarrow H_{1, \tau}^{dR}\rightarrow \omega_{\bar{B}, \tau}^*\rightarrow 0.
\]
The Hodge filtration is defined  in \cite[\S~5.1]{katotrihan}, and proved to be locally free in Lemma 5.2 of loc.cit.. The proof that the sub and quotient are isomorphic to  $\omega_{\bar{B}^t, \tau}, \omega_{\bar{B}, \tau}^*$, respectively, is given in \cite[Example 5.4 (b)]{katotrihan}.
%where by definition, $\omega_{\bar{B}^t, \tau} \subset H_{1, \tau}^{dR}$ is the kernel of $F$.
We sometimes also denote the sub-bundle $\omega_{\bar{B}^t, \tau}$ by $\Fil^1_{\tau}$. The Frobenius action on $H^{dR}_{1}(B)$ annihilates each of $\Fil^1_{\tau}$
\end{enumerate}
\end{defnprop}
\begin{prop}\label{prop:dim1}
We have $\dim H_{1, \tau}^{dR}=2$, $\dim \Fil^1_{\tau}=1$ for all embeddings $\tau$.
\end{prop}
\begin{proof}
For each  point $c: \Spec(\Fpbar)\rightarrow  C$,  let $H_{1}^{cris}(B_c)$ denote the crystalline homology of the fiber $B_c$. Since $B$ is of $\GL_2(\mc{O}_F)$-type, we have $\dim H_{1, \tau}^{cris}(B_c)=2$, where $H_{1, \tau}^{cris}(B_c)$ denotes the summand on which the $\mc{O}_F$-action is through the embedding $\mc{O}_F\rightarrow W(\Fpbar)$ lifting $\tau$. Therefore $\dim H^{dR}_{1, \tau}=2$ for each $\tau$, and it remains to show that $\dim \Fil^1_{\tau}=1$.

It suffices to show this when restricted to a neighborhood of a cusp $z\in Z$. Let $\pdiv$ denote the $p$-divisible group of $B$, and let $k[[t]]$ denote the local ring around $z$. The totally degenerate assumption implies $\pdiv|_{\Spec(k((t)))}$ has a sub $p$-divisible group $\mscr{H}$ of multiplicative type and dimension $d$. Let $\kappa\defeq k((t^{1/p^{\infty}}))$ denote the perfection of $k((t))$, and consider the inclusion of Dieudonn\'e modules  $\mc{D}(\mscr{H})\subset \mc{D}(\pdiv|_{\Spec(\kappa)})$. We denote by $H_1^{dR}(\mscr{H})$ the reduction mod $p$ of $\mc{D}(\mscr{H})$, which we may take as the definition of the de Rham homology of $\mscr{H}$. %as well as the $\tau$-components $\mc{D}(\mscr{H}|_{\Spec(\kappa)})_{\tau}$. 

Recall that Frobenius annihilates $\Fil^1\subset H^{dR}_{1}(B|_{\Spec(\kappa)})$. Since $\mscr{H}$ is of multiplicative type, the  Frobenius action is bijective (as we are using covariant Dieudonn\'e theory) on $H_1^{dR}(\mscr{H})$,  the induced map 
\[
H^{dR}_1(\mscr{H})\rightarrow H^{dR}_1(B|_{\Spec(\kappa)})/\Fil^1=\omega^*_{B|_{\Spec(\kappa)}}
\]
is injective, and hence bijective for dimension reasons. Now $\mscr{H}$ is acted on by $\mc{O}_F\otimes \mb{Z}_p=\mc{O}_{\mf{p}_1}\times \cdots \times \mc{O}_{\mf{p}_k}$, where $\mf{p}_i$ are the primes of $F$ lying over $p$, and $\mc{O}_{\mf{p}_i}$ denotes the completion of $\mc{O}_F$ at $\mf{p}_i$, which by our assumption is unramified over $\mb{Z}_p$; hence we have a corresponding decomposition of $p$-divisible groups
\[
\mscr{H}=\mscr{H}_1\times \cdots \times \mscr{H}_k,
\]
where $\mscr{H}_i$ is the image of the idempotent corresponding to the $\mc{O}_{\mf{p}_i}$ factor. By construction, $\mscr{H}_i$ is a $p$-divisible group of $\GL_2(\mc{O}_{\mf{p}_i})$-type. Recall that each $\tau$ specifies one of the  primes $\mf{p}_i$, and an embedding of the residue field $k_{\mf{p}_i}\xhookrightarrow{} \mb{F}_q$. Of course, fixing such an  embedding  $\tau: k_{\mf{p}_i}\xhookrightarrow{} \mb{F}_q$, every other embedding is of the form $\sigma^m\tau$, the composition of $\tau$ with  the $m$-th power of   absolute Frobenius on $\mb{F}_q$.

Each of $\mscr{H}_i$ is necessarily of multiplicative type, and hence the Frobenius 
\[
F_i: H_1^{dR}(\mscr{H}_i)^{(p)}\rightarrow  H_1^{dR}(\mscr{H}_i)
\]
is again bijective. %here, by $H_1^{dR}(\mscr{H}_i)$ we simply mean the reduction mod $p$ of the Dieudonn\'e module $\mc{D}(\mscr{H}_i)$. 
Let  $H_{1}^{dR}(\mscr{H}_i)=\bigoplus_{\tau} H_{1, \tau}^{dR}(\mscr{H}_i)$ denote the decomposition into $\tau$-components, as $\tau$ varies over embeddings $k_{\mf{p}_i}\xhookrightarrow{} \mb{F}_q$. Now, for each $\tau$, $F_i$ sends $H_{1, \tau}^{dR}(\mscr{H}_i)$ to $H_{1, \sigma \tau}^{dR}(\mscr{H}_i)$; as $\mscr{H}_i$ is of multiplicative type, this is an isomorphism. In other words, for a fixed $i$, the dimension of $H_{1, \tau}^{dR}(\mscr{H}_i)$ is independent of $\tau$, and hence  at least one. Since $H^{dR}_1(\mscr{H})$ has dimension $d$, we deduce that $\dim H_{1, \tau}^{dR}(\mscr{H}_i)=1$ for each choice of $i$ and $\tau$. This implies that $\omega^*_{\bar{B}, \tau}$ has rank one for each $\tau$, and hence the same is true of $\Fil^1_{\tau}$, as required.

%Similarly, by considering the \'etale part of $\pdiv$, we see that the Newton polygon has $d$ slope one segments and $d$ slope zero ones, i.e.  it is ordinary. On the other hand, the kernel of the Frobenius action on $H^{dR}_1$ is precisely $\omega_{B^t}$, which has dimension $d$. 

%The Frobenius action is bijective on the latter, and since $F$ annihilates $\dim \Fil^1_{\tau}$, we must have $\dim \Fil^1_{\tau}\leq 1$. On the other hand, $\Fil^1_{\tau}\cong \omega_{\bar{B}^t, \tau}$, and we conclude $\dim \Fil^1_{\tau}=1$ by a dimension count.
\end{proof}

We now recall the notion of partial Hasse invariants, following  \cite[\S~4.4]{tianxiao}. For any abelian scheme   $A$ over a characteristic $p$ scheme $S$, we denote by $A^{(p)}$ the pull-back of $A$ along absolute Frobenius on $S$; similarly for any coherent sheaf $M$ on $S$ we have the Frobenius pull-back $M^{(p)}$. Let $\sigma$ be absolute Frobenius on $\mb{F}_q$, and denote  by $\sigma \tau$ the composition $k_{\mf{p}}\xhookrightarrow{} \mb{F}_q \xrightarrow{\sigma} \mb{F}_q$; let $\sigma^{-1}\tau$ be the embedding such that $\sigma(\sigma^{-1}\tau)= \tau$. %If $S$ is a curve $C$ with compactification $\bar{C}$, and $A$ has semi-stable reduction along $\bar{C}-C$, with N\'eron model $\bar{A}\rightarrow \bar{C}$, then  

\begin{defn}
Notation as in \ref{setup:avB}. For each $\tau$, we have the  map $\omega_{\bar{B}^t, \tau} \rightarrow \omega_{\bar{B}^{t, (p)}, \tau}=\omega_{\bar{B}^t, \sigma^{-1}\tau}^{\otimes p}$ induced  by Verschiebung. We view this as a section $h_{\tau}\in \Gamma(\omega_{\bar{B}^t, \sigma^{-1}\tau}^{\otimes p} \otimes \omega_{\bar{B}^t, \tau}^{\otimes -1})$, and refer to it as a partial Hasse invariant.
\end{defn}

The following is a simple by-product of the proof of \Cref{prop:dim1}, and we therefore omit the proof.
\begin{prop}\label{prop:hassenonzero}
%Suppose $A\rightarrow C$ is an abelian scheme of $\GL_2(F)$-type, and such that it has totally degenerate reduction around $z\in \bar{C}-C$. 
Each of the $h_{\tau}$'s is not identically zero. 
\end{prop}
\iffalse 
\begin{proof}
Let $\pdiv$ denote the $p$-divisible group $A[p^{\infty}]$, and let $k[[t]]$ denote the local ring around $z\in \bar{C}$. %By the totally degenerate reduction assumption,   $\pdiv|_{k((t))}$ has a multiplicative sub $p$-divisible group of dimension $d$; let us denote this by $\mscr{H}$. Now let $k((t))^{\mathrm{perf}}\defeq k((t^{1/p^{\infty}}))$ be the perfection of $k((t))$, and  consider the Dieudonn\'e modules 

As in the proof of Proposition~\ref{prop:dim1}, consider
\[
\mc{D}(\mscr{H})\subset \mc{D}(\pdiv|_{\Spec(\kappa)}).
\]
For brevity let us denote the former by $N$ and the latter by $M$. We have the usual decomposition $M=\bigoplus M_{\tau}$, where each $M_{\tau}$ is free of rank two over $W(\kappa)$, and moreover there is a non-trivial Hodge filtration $\Fil^1_{\tau} \subset M_{\tau}/pM_{\tau}$.  

Since $\mscr{H}$ is of multiplicative type, Frobenius is an isomorphism on $N/pN$, and hence $N_{\tau}/pN_{\tau}\subset M_{\tau}/pM_{\tau}$ is one-dimensional, and which is moreover complementary to $\Fil^1_{\tau}$. Therefore the image of 
\[
F: (M_{\sigma^{-1}\tau}/pM_{\sigma^{-1}\tau})^{(p)} \rightarrow M_{\tau}/pM_{\tau} 
\]
is not contained in $\Fil^1_{\tau}$, and therefore $h_{\tau}$ is not identically zero by Proposition~\ref{prop:hassecrit}. 
\end{proof}
\fi
\section{Frobenius untwisting}\label{section:untwist}
\begin{defn}
Let $K/\mb{Q}_p$ be a local field with ring of integers $\mc{O}_K$. We say that a $p$-divisible group $\pdiv$ is of $\GL_2(\mc{O}_K)$-type if $\mc{O}_K\xhookrightarrow{} \End(\pdiv)$, and moreover $[K:\mb{Q}_p]=\dim \pdiv$.
\end{defn}
%Let $\pdiv$ be a $p$-divisible group of $\GL_2(\mc{O}_K)$-type over $C/\mb{F}_q$, for some $K/\mb{Q}_p$ which is unramified. 
We prove the following result, which generalizes a result of Xia's \cite[Theorem 6.1]{xia}. As before, suppose $C/\mb{F}_q$ is a smooth affine curve. For the definition and notation surrounding the Kodaira--Spencer map of a family of $p$-divisible groups, we refer the reader to \cite[\S 1.1]{xia}, as well as Theorem 3.11 of loc.cit.. For any sheaf $\mscr{F}$ on $C$ (e.g. a vector bundle or a $p$-divisible group), we write $\mscr{F}^{(p)}$ for the Frobenius twist, i.e. the pullback under absolute Frobenius on $C$.

\begin{lem}\label{lemma:untwist}
Let $\pdiv$ be any $p$-divisible group over $C$, such that  the Kodaira-Spencer map of $\pdiv$ is identically zero. Then there exists a $p$-divisible group $\mscr{H}$ isogenous to $\pdiv$ and such that $\mscr{H}^{(p)}\cong \pdiv$. Moreover, if $\pdiv$ is of $\GL_2(\mc{O}_K)$-type,  then so is $\mscr{H}$.
\end{lem}
\begin{proof}
Suppose $\pdiv$ has dimension $d$ and height $h$. Let $\mc{V}$ be the Dieudonn\'e crystal of $\pdiv$, and let $\mc{V}_C$ denote its evaluation on $C$, which is a vector bundle on $C$ of dimension $h$; moreover, it has a sub-bundle $\Fil^1$ given by the Hodge filtration, as well as the usual  Frobenius and Verschiebung maps
\[
F_C: \mc{V}_C^{(p)} \rightarrow \mc{V}_C, \ V_C: \mc{V}_C\rightarrow \mc{V}_C^{(p)}.
\]

Fix any lifting $\tilde{C}$ of $C$ over $W(\mb{F}_q)$, along with a lift $\sigma$ of absolute Frobenius. Evaluating the crystal on $\tilde{C}$ we obtain a vector bundle on $\tilde{C}$ with integrable connection, which we  denote by $(\mc{V}_{\tilde{C}}, \nabla)$; this has its own Frobenius and Verschiebung, which we denote by $F_{\tilde{C}}, V_{\tilde{C}}$ respectively; moreover we have  a mod $p$ reduction map $\pi: \mc{V}_{\tilde{C}}\rightarrow \mc{V}_C$.

We now construct the Dieudonn\'e crystal for the hypothetical \enquote{Frobenius untwisting} of $\mscr{G}$. Let $\mc{V}'_{\tilde{C}}\subset \mc{V}_{\tilde{C}}$ denote the sub-module $\pi^{-1}(\Fil^1)$, which contains $p\mc{V}_{\tilde{C}}$; from the conditions  $(\Fil^1)^{(p)}=\ker F_C=\Imag V_C$ it is straightforward to see that $\mc{V}'_{\tilde{C}}$ is stable under both $F_{\tilde{C}}$ and $V_{\tilde{C}}$, as well as the $\mc{O}_K$-action if $\pdiv$ is of $\GL_2(\mc{O}_K)$-type. Moreover, since we are assuming that  Kodaira-Spencer map is identically zero, $\mc{V}'_{\tilde{C}}$ is stable under $\nabla$, and therefore $\mc{V}'_{\tilde{C}}$ is the Dieudonn\'e module for a $p$-divisible group $\mscr{H}$, which is  of $\GL_2(\mc{O}_K)$-type if $\pdiv$ is.   

%For any affine open $U\subset \tilde{C}$, let $\sigma$ be a lift of the absolute Frobenius.
We claim that  
\begin{equation}\label{eqn: untwist}
F_{\tilde{C}}(\mc{V}_{\tilde{C}}'^{(p)})=p\mc{V}_{\tilde{C}}:
\end{equation} indeed, by the construction of $\mc{V}'$, we have the inclusions
\[
F_{\tilde{C}}(\mc{V}_{\tilde{C}}'^{(p)})\subseteq p\mc{V}_{\tilde{C}}\subseteq \mc{V}_{\tilde{C}}'\subseteq \mc{V}_{\tilde{C}}.
\]
It remains to check that the left most inclusion is an equality.  We have $\dim\mc{V}_{\tilde{C}}'/F_{\tilde{C}}(\mc{V}_{\tilde{C}}'^{(p)})=h-d$, and  also $\dim \mc{V}_{\tilde{C}}'/p\mc{V}_{\tilde{C}}=h-d$, where dimension refers to the dimension as vector bundles on $C$; these, together with the chain of inclusions $F_{\tilde{C}}(\mc{V}'^{(p)})\subseteq p\mc{V}_{\tilde{C}}\subseteq \mc{V}_{\tilde{C}}'$, prove the claim.  But \eqref{eqn: untwist} says precisely that $\mscr{H}^{(p)}\cong \pdiv$; finally, the construction gives a natural isogeny $\mscr{H}\rightarrow \pdiv$, as required.
\end{proof}
\iffalse 
\begin{prop}\label{prop:negdegree}
Suppose each of  the $h_{\tau}$'s  is not identically zero. Then $\deg \Lie(A) \leq 0$. 
\end{prop}
\begin{proof}
By the assumption on $h_{\tau}$, the restrcting of essential Frobenius  gives a non-zero map
\[
F_{es}^{f}: \Lie(A^{(p^f)})_{\tau} \rightarrow \Lie(A)_{\tau},
\]
where $f$ is the degree of the residue field of $\mf{p}$. In other words, $\deg \Lie(A)_{\tau}^{1-p^f}$ has a non-zero section, and hence $\deg \Lie(A)_{\tau}\leq 0$.
\end{proof}
\fi 
\begin{proof}[Proof of Proposition~\ref{prop:drinfeld}]
Let $B\defeq B_{Drinf}\rightarrow C$ be the abelian scheme given by Theorem~\ref{thm:drinf}, which is of $\GL_2(F)$-type, and we assume that $\mc{O}_F$ acts on $B_{Drinf}$ as in \ref{setup:avB}. Let $\pdiv$ denote the $p$-divisible group of $B$; for  a prime $\mf{p}$  of $F$ lying above $p$, let $\pdiv_{\mf{p}}$ denote the factor of $\pdiv$ such that the $\mc{O}_F$-action is through the embeddings corresponding to $\mf{p}$ (exactly as in the proof of \Cref{prop:dim1}). By applying  Lemma~\ref{lemma:untwist} repeatedly, which must terminate by the same argument as in  \cite[Proof of Proposition 2.4]{zuoconstruct}\footnote{another proof of this termination is  given in \cite[Appendix A]{zuoconstruct}, which also applies essentially without change in our set-up}, we may assume that the Kodaira-Spencer map for $\pdiv_{\mf{p}}$ is non-vanishing. Let $\tau_0$ be an embedding for the prime $\mf{p}$ for which the Kodaira-Spencer map
\[
\Fil^1_{\tau_0} \rightarrow H_{1, \tau_0}^{dR}/\Fil^1_{\tau_0} \otimes \Omega_{\bar{C}}^1(Z)
\]
is non-zero: note that this is a map of line bundles, by \Cref{prop:dim1}. Therefore   $\deg H_{1, \tau_0}^{dR}/\Fil^1_{\tau_0}=-\deg \Fil^1_{\tau_0}$, from which we may conclude that  $\deg \Fil^1_{\tau_0}\leq (2g-2+n)/2$. 

By Proposition~\ref{prop:hassenonzero}, for any other embedding $\tau'=\sigma^j \tau_0$ for the prime $\mf{p}$  for some $j\geq 0$, we have a non-zero map 
\[
\omega_{\sigma^j\tau_0}\rightarrow \omega_{\tau_0}^{\otimes p^j}
\]
by iterating Verschiebung, and hence $\deg \omega_{\sigma^j \tau_0}\leq p^j\deg \omega_{\tau_0}$. This implies that, for each $\tau$,  $\deg \omega_{\tau}$ is bounded above purely in terms of $g,n, p, d$. Finally, recall that $\Fil^1_{\tau}\cong \omega_{\bar{B}^{t}, \tau}$ and $H^{dR}_{1, \tau}/\Fil^1_{\tau}\cong \omega_{\bar{B}, \tau}^*$, and therefore  the Hodge bundle of the  abelian scheme $(B\times B^t)^4$  has degree bounded above purely in terms of $p,g,n,d$, as required.

\end{proof}

\section{Mapping spaces}
For any $N\geq 1$, let $\mscr{A}_N$ denote the moduli stack of principally polarized abelian varieties of dimension $N$, with minimal, i.e. Baily-Borel-Satake, compactification $\mscr{A}_N^*$. Let $\omega$ denote the Hodge bundle on $\mscr{A}_N^*$.
\begin{defn}
Let $\mc{M}_{g,n}(\mscr{A}_N^*, h)$ denote the moduli stack\footnote{similar mapping stacks have been studied by Faltings\cite[\S~3]{faltings}} of $n$-pointed genus $g$ curves $(\bar{C}, p_1, \cdots , p_n)$, along with a map $\bar{f}: \bar{C}\rightarrow \mscr{A}_N^*$ such that 
\begin{enumerate}
    \item $\bar{f}^{-1}(\partial \mscr{A}_N^*)\subset \{p_1, \cdots, p_n\}$, and 
    \item $\bar{f}^*\omega$ has degree $h$.
\end{enumerate}
\end{defn}
Let $\mc{S}$ be the set of trace fields defined  in the introduction, and let $\mc{S}_{unip}\subset \mc{C}$ denote the subset of trace fields coming from local systems with unipotent local monodromy around each cusp of $C$.
\begin{lem}\label{lemma:family}
Suppose there are infinitely many fields $F_1, F_2, \cdots $ in $\mc{S}_{unip}$. Then there exists a smooth scheme of finite type $S/\mb{F}_q$, a family of smooth curves $\mc{C}\rightarrow S$, an abelian scheme $\mc{A}\rightarrow \mc{C}$, and a Zariski dense set of points $s_i\in S$ such that the family $\mc{A}_{s_i}\rightarrow \mc{C}_{s_i}$ is isomorphic to $(\mc{B}_{s_i}\times \mc{B}_{s_i}^t)^4$, where $\mc{B}_{s_i}\rightarrow \mc{C}_{s_i}$ is of  $\GL_2(F_i)$-type.
\end{lem}

\begin{proof}
By Proposition~\ref{prop:drinfeld}, for each $i$, there is a pointed curve   $(\bar{C}_i, Z_i)$  of type $(g,n)$,  a map 
\[
\bar{f}_i: \bar{C}_i\rightarrow \mscr{A}_N^*
\]
with $\bar{f}_i^{-1}\partial \mscr{A}_N^*\subset Z_i$, and such that the degree $\bar{f}_i$ is bounded in terms of $g,n,p,d$.
\iffalse 
Now consider the moduli stack $\mc{M}_{g,n}(\mscr{A}_N^*, h)$, which comes with the forgetful map 
\[
\pi: \mc{M}_{g,n}(\mscr{A}_N^*,h)\rightarrow \mc{M}_{g,n}.
\]
By the above, the image of $\pi$ has infinitely many points, which we denote by $s_i\in \mc{M}_{g,n}$. Since $\mc{M}_{g,n}(\mscr{A}_N^*, h)$ is of finite type, the image of $\pi$ has positive dimension. 
\fi 
Therefore, for some $h$, we have infinitely many points $s_i\in \mc{M}_{g,n}(\mscr{A}_N^*, h)$. Since the latter is of finite type, taking $S$ to be the smooth part of the Zariski closure of the $s_i$ gives the desired family of curves $\mc{C}\rightarrow S$ and abelian scheme $\mc{A}$.
\end{proof}

\section{Proofs of main results}
\begin{prop}\label{prop: trivial-torsion-implies-tame-unip}
    Suppose $k$ is an algebraically closed field of characteristic $p>0$, $Z/k$ a smooth scheme of finite type, and $Z\subset \bar{Z}/k$ is a compactification, i.e. $\bar{Z}$ is normal, projective, and $D:= \bar{Z}\setminus Z$ is a simple normal crossings divisor. If $\pi_{\mc{A}}: \mc{A}\rightarrow Z$ is an abelian scheme such that the action of $\pi_1(Z)$ on the $\ell$-torsion of $\mc{A}$ is trivial, for $\ell>2$ a prime number, then $R^1\pi_{\mc{A}}\Qlbar$ has unipotent (geometric)  monodromy   along each  component of the boundary $D$. Moreover, the image of the representation of $\pi_1(Z)$ has pro prime-to-$p$ image.
\end{prop}
\begin{proof}
The claim about unipotent monodromy is given on \cite[p. 879]{krishnapal} in the third paragraph. For the prime-to-$p$ claim, note that, for $\ell>2$,  the subgroup $\Gamma(\ell)\subset \GL_n(\mb{Z}_{\ell})$, consisting of matrices which are the identity mod $\ell$, is pro-$\ell$, as long as $\ell >2$.
\end{proof}

\begin{lem}\label{lemma: stable-all-etale}
    Let $C/\mb{F}_q$ be a smooth curve, and $\pi: B\rightarrow C$ an abelian scheme of $\GL_2(F)$-type, for a number field $F$. Let $\mb{W}:=R^1\pi_*\mb{Q}_{\ell}$ be the relative Tate module of $B$, so that we have a decomposition 
    \[
    \mb{W}\otimes \Qlbar \cong \bigoplus_{\sigma: F\xhookrightarrow{} \Qlbar}\mb{L}_{\sigma}
    \]
    with each $\mb{L}_{\sigma}$ being of rank two.
    \begin{enumerate}
    \item  For any finite \'etale cover $C'\rightarrow C$, and any choice of $\sigma$, the trace field of $\mb{L}_{\sigma}|_{C'}$ is still $F$. For $\sigma\neq \tau$, the local systems $\mb{L}_{\sigma}, \mb{L}_{\tau}$ are not isomorphic.
\item 
Moreover, for any finite index subgroup $H\leq \pi_1(C_{\Fpbar})$, 
\[
\End(\mb{W}\otimes \Qlbar)^{H}=F\otimes \Qlbar=\prod_{\sigma: F\xhookrightarrow{} \Qlbar} \Qlbar.
\]
\end{enumerate}
\end{lem}

\begin{proof}
   We first prove point (1). Note that the second assertion of point (1) follows immediately from the first: indeed, if $\sigma\neq \tau$ and  $\mb{L}_{\sigma}|_{C'}\cong \mb{L}_{\tau}|_{C'}$, then $\sigma$ and $\tau$ must agree on the trace field of $\mb{L}_{\sigma}|_{C'}$, which means this latter trace field is strictly smaller than $F$. 
   
   We now prove the first assertion of (1). Suppose the contrary, i.e. that the trace field of $\mb{L}_{\sigma}|_{C'}$ is a strict sub-field $E\subset F$, for some $\sigma$. Taking a different embedding $\tau: F\xhookrightarrow{} \Qlbar$ which agrees with $\sigma$ on $E$, we get 
    \begin{equation}\label{eqn: etale-isom}
    \mb{L}_{\sigma}|_{C'} \cong \mb{L}_{\tau}|_{C'}.
    \end{equation}

    We may assume without loss of generality that $C'$ is Galois over $C$, so that $\pi_1(C)/\pi_1(C')$ is a finite group $\Gamma$. Since $\mb{L}_{\sigma}|_{C'}, \mb{L}_{\tau}|_{C'}$ are irreducible, \eqref{eqn: etale-isom} implies that there exists a character  $\chi: \Gamma \rightarrow \Qlbar^{\times}$ such that $\mb{L}_{\sigma}\cong \mb{L}_{\tau}\otimes \chi$. Since $\det \mb{L}_{\sigma}\cong \Qlbar(-1) \cong \det \mb{L}_{\tau}$, we deduce that $\chi$ factors as 
    \[
    \chi: \Gamma\rightarrow \{\pm 1\} \rightarrow \Qlbar^{\times}.
    \]
    Replacing $C'$ by the cover of $C$ defined by $\chi$., we may therefore assume that $C'$ has degree two over $C$. Suppose that the trace field of $\mb{L}_{\sigma}|_{C'}$ is a number field $E\subset F$.

    Applying Drinfeld's result (Theorem~\ref{thm:drinf}) to    $\mb{L}_{\sigma}|_{C'}$, we see that  there is some abelian scheme $B'\rightarrow C'$ of dimension $[E:\mb{Q}]$, and of $\GL_2(E)$-type, so that 
\[
\mb{W}|_{C'} \cong \mb{W}(B')^{\oplus m},
\]
where $\mb{W}(B')$ denotes the $\mb{Q}_{\ell}$-Tate module of $B'$, and  $m\defeq [F:E]$.

Consider the abelian scheme $\Res^{C'}_C B'$ over $C$. By the definition of Weil-restriction and the Tate conjecture for homomorphisms of abelian varieties over function fields (proven by Zarhin), we have 
\begin{equation}\label{eqn:dimhom2}
\dim_{\mb{Q}} \Homo_C(B, \Res^{C'}_C B')\otimes \mb{Q}=m,
\end{equation}
which is strictly greater than one by our assumption.

On the other hand, $B$ is simple over $C$, and since $C'$ is a degree two cover of $C$, $\dim \Res^{C'}_C B'=2[E:\mb{Q}]\leq m[E:\mb{Q}]=\dim B$. This contradicts Equation~\ref{eqn:dimhom2}.

    We now deduce the second part. For $H$ a finite index subgroup of $\pi_1(C_{\Fpbar})$, let 
    \begin{equation}\label{eqn: cover}
    \tilde{C}_{\Fpbar}\rightarrow C_{\Fpbar}
    \end{equation}
    be the corresponding finite \'etale cover of $C_{\Fpbar}$. Note that  there exists a finite extension  $\mb{F}_{q'}$ over which  \eqref{eqn: cover} is defined over: i.e. a curve $\tilde{C}/\mb{F}_{q'}$, along with a map $\tilde{C}\rightarrow C_{\mb{F}_{q'}}$ whose basechange along $\mb{F}_{q'}\xhookrightarrow{} \Fpbar$ is the map \eqref{eqn: cover}.

    We have the natural embedding $F\otimes \Qlbar \xhookrightarrow{}\End(\mb{W}\otimes \Qlbar)^H$, and suppose for the sake of contradiction that this is a strict inclusion.

This implies that two of the $\mb{L}_{\sigma}$'s become isomorphic when viewed as representations of $H$: more precisely, there are distinct embeddings $\sigma, \tau: F\xhookrightarrow{} \Qlbar$, such that $\mb{L}_{\sigma}|_{\tilde{C}_{\Fpbar}}\cong \mb{L}_{\tau}|_{\tilde{C}_{\Fpbar}}$. This in turn implies that there is a character $\chi:\pi_1(\Spec(\mb{F}_{q'}))\rightarrow \Qlbar^{\times}$ such that 
\[
\mb{L}_{\sigma}|_{\tilde{C}}\cong \mb{L}_{\tau}|_{\tilde{C}} \otimes \chi.
\]
As before, comparing determinants, we deduce that $\chi$ factors through $\{\pm 1\}\subset \Qlbar^{\times}$; therefore $\chi$ determines a degree two \'etale cover $\tilde{C}'\rightarrow \tilde{C}$ over which $\mb{L}_{\sigma}$ and $\mb{L}_{\tau}$ are isomorphic. This contradicts the first part of the lemma.
\end{proof}

\begin{comment} \begin{prop}
    Let $C/\mb{F}_q$ be a smooth curve, $B\rightarrow C$ an abelian scheme of $\GL_2(F)$-type. Suppose $\mc{G}\rightarrow C$ is another abelian scheme such that  the following condition holds: for any finite extension $\mb{F}_{q'}/\mb{F}_q$, write $C':=\mc{C}_{s_i}\times \Spec(\mb{F}_{q'})$; then  there are no non-zero maps of abelian schemes 
    \[
   B\times C'\rightarrow \mc{G}\times C'\]
    over $C'$. 

    For a prime $\ell \neq p$, let $V_A, V_B, V_{\mc{G}}$ be the $\Qlbar$-representations of $\pi_1(C)$ on the relative Tate modules of $A, B$ and  $\mc{G}$ respectively. Then 
    \[\End(V_A)^{\pi_1(\Cbar)}=\End(V_B)^{\pi_1(\Cbar)}\times \End(V_{\mc{G}})^{\pi_1(\Cbar)}.
    \]
\end{prop}

\begin{proof}
    It suffices to show that 
    \[
    \Hom_{\pi_1(\Cbar)}(V_B, V_{\mc{G}})=0.
    \]
Suppose there is a non-zero element $\Phi$. There exists $C':=\mc{C}_{s_i}\times \Spec(\mb{F}_{q'})$ for some finite extension $\mb{F}_{q'}/\mb{F}_q$, such that  there are positive dimensional sub-representations $U_B\subset V_B$, $U_{\mc{G}}\subset V_{\mc{G}}$ (i.e. stable under $\pi_1(C')$) such that 
\begin{itemize}
    \item $U_B|_{\pi_1(\Cprimebar)}, U_{\mc{G}}|_{\pi_1(\Cprimebar)}$ are irreducible,
    \item $\Phi$ restricts to an isomorphism \[
    \Psi: U_B|_{\pi_1(\Cprimebar)}\cong U_{\mc{G}}|_{\pi_1(\Cprimebar)}.
    \]
\end{itemize}
\end{proof}
\end{comment}

\begin{prop}\label{prop: prime-to-p-iso}
    Let $(\mc{C}, \mc{D})\rightarrow S$ be a family of smooth pointed curves: that is, there exists a smooth proper $S$-curve $\overline{\mc{C}}$, $\mc{D}$ is a relatively \'etale Cartier divisor on $\overline{\mc{C}}/S$, and $\mc{C}=\overline{\mc{C}}\setminus \mc{D}$. Suppose  $S$ is geometrically connected.  Suppose there is an abelian scheme $\mc{A}\rightarrow \mc{C}$, such that the action of $\pi_1(\mc{C})$ on the $\ell$-torsion is trivial. Then 
    \begin{itemize}
        \item for any closed point $s\in S$ with geometric point $\bar{s}$ lying over it, the representation of $\pi_1(\mc{C}_{s})$ on the $\ell$-adic Tate module of $\mc{A}|_{\mc{C}_{s}}$ factors through the prime-to-$p$ quotient $\pioneprime(\mc{C}_{s})$, and 
        \item for any other  closed point $t\in S$, with geometric points $\bar{t}$ lying over it, the representations of $\pioneprime   (\mc{C}_{\bar{s}})$ and $\pioneprime(\mc{C}_{\bar{s}})$ (on the $\ell$-adic Tate modules of $\mc{A}|_{\mc{C}_{\bar{s}}}$ and $\mc{A}|_{\mc{C}_{\bar{t}}}$, respectively) are isomorphic: that is, there exists a continuous  isomorphism $\pioneprime(\mc{C}_{\bar{s}})\cong \pioneprime(\mc{C}_{\bar{t}})$ such that these representations are isomorphic.
    \end{itemize}
\end{prop}

\begin{proof}
The first part follows from  \Cref{prop: trivial-torsion-implies-tame-unip}. We now prove the second part. Pick a point $s_{gen}$ of $S$ specializing to both $s$ and $t$, and let $\bar{s}_{gen}$ be a geometric point lying over $s_{gen}$. 

We have specialization isomorphisms of prime-to-$p$ fundamental groups (see \cite[Expos\'e XIII, Corollaire 2.12]{sga1960grothendieck} for this precise statement; we also refer the reader to \cite[\S 3]{otabe2018tame} for a nice exposition):
\[
\sp_{s}: \pioneprime(\mc{C}_{\bar{s}_{gen}}) \cong \pioneprime(\mc{C}_{\bar{s}}), \ \sp_t: \pioneprime(\mc{C}_{\bar{s}_{gen}}) \cong \pioneprime(\mc{C}_{\bar{t}}).
\]
By the construction  of the specialization map and proper base-change, the representation of $\pioneprime(\mc{C}_{\bar{s}})$ on the $\ell$-adic Tate module of $\mc{A}|_{\mc{C}_{\bar{s}}}$ is obtained by taking the corresponding representation of $\mc{C}_{\bar{s}_{gen}}$ and composing with the isomorpshim $\sp_s^{-1}$. The same goes for the representation of $\pioneprime(\mc{C}_{\bar{t}})$, which implies the claim.  
\end{proof}

The following is the key lemma to show finiteness of trace fields.

\begin{lem}\label{lemma:center}
Let $(\mc{C}, \mc{D})\rightarrow S$ be a family of smooth pointed curves.  Suppose there is an abelian scheme $\mc{A}\rightarrow \mc{C}$, and an infinite set of points $s_i\in S$ such that, for each $i$,  the  abelian scheme
$\mc{A}_{s_i}\rightarrow \mc{C}_{s_i}$
is isogenous to $B_i^{k}$ for an abelian scheme $B_i\rightarrow \mc{C}_{s_i}$ which is of $\GL_2(F_i)$-type, for some number field $F_i$, and $k\geq 1$ (the latter being independent of $i$). Then the set $\{F_i\}$ of trace fields  is finite.
\end{lem}

\begin{proof}
Since passing to finite \'etale covers of $\mc{C}_{s_i}$ does not change the trace field by \Cref{lemma: stable-all-etale}, we may as well assume that the action of $\pi_1(\mc{C})$ on the $\ell$-torsion of $\mc{A}$ is trivial. Under this assumption, by \Cref{prop: prime-to-p-iso}, for each $s_i$, the representation $\mb{V}_{\ell, \bar{s}_i}$ of $\pi_1(\mc{C}_{\bar{s}_i})$ on the $\ell$-adic Tate module of $\mc{A}|_{\mc{C}_{\bar{s}_i}}$ factors through the prime-to-$p$ quotient $\pioneprime(\mc{C}_{\bar{s}_i})$, and moreover the resulting representations of $\pioneprime(\mc{C}_{\bar{s}_i})$ are isomorphic for all $i$, in the sense of the statement of \Cref{prop: prime-to-p-iso}.

 Fix $s=s_1$, and let $\mb{V}_{\ell}$ denote the $\ell$-adic Tate module of $\mc{A}_{s}$. As $\ell$ varies over primes different from $p$, the $\mb{V}_{\ell}$'s form a $\mb{Q}$-compatible system\footnote{see the introduction of \cite{chin} for the definition of compatible systems}. Let $G_{\mathrm{geom}, \ell}$ denote the Zariski closure of the image of the geometric fundamental group $\pi_1(\mc{C}_{\bar{s}_i})$ under this representation; we view $\mb{V}_{\ell}$ as a representation of $G_{\mathrm{geom}, \ell}$. Let $G_{geom, \ell}^0$ denote the neutral component of $G_{geom, \ell}$. By the previous paragraph, the isomorphism class of the triple $(G_{geom, \ell}, G_{geom, \ell}^0, \mb{V}_{\ell})$ is independent of $i$, and we therefore chose to  omit   $i$ from the notation. 

By Chin's independence of $\ell$ theorem \cite[Theorem~1.4]{chin}, there exists a number field $E$, a split algebraic group $G^0/E$, and a $E$-rational representation $\mb{U}$ such that, for any place $\lambda$ of $E$ above $\ell$,  $G^0_{geom, \ell}\otimes E_{\lambda}\simeq G^0_{E_{\lambda}}$, and  $\mb{V}_{\ell}\otimes_{\mb{Q}} E_{\lambda}\simeq \mb{U}\otimes E_{\lambda}$ as representations of these groups. Define  $Z\defeq Z(\End(\mb{U})^{G})$, which is a finitely generated  $E$-algebra.

%Consider the $\mb{Q}$-local system on $\mc{C}_{s_{i, \mb{C}}}$ given by $\mb{V}=R^1\tilde{\pi}_{i, \mb{C} *}\mb{Q}.$ Note that, by isomonodromy,  the isomorphism class of the pair $(\pi^{top}_1(\mc{C}_{s_{i, \mb{C}}}), \mb{V})$ is independent of $i$. Now consider $Z\defeq Z(\End(\mb{V})^{\pi^{top}_1(\mc{C}_{\tilde{s}_{i, \mb{C}}})})$, where  $Z(R)$ denotes the center of an algebra $R$.

\begin{claim} 
We have $F_i\otimes \mb{Q}_{\ell}\xhookrightarrow{} Z\otimes_E E_{\lambda}$ for each $i$ and $\lambda$.
\end{claim} 
The claim implies the finiteness of the $F_i$'s: indeed, $Z$ is a product of number fields, so is unramified at $\ell$ for $\ell$ large enough. Therefore the $F_i$'s belong to the finite set \cite[Ch. II Theorem 2.13]{neukirch2013algebraic} of number fields with bounded degree and unramified outside a finite set of primes.
\begin{proof}[Proof of Claim] 
%Recall that  the family $\mc{A}_{s_i}\rightarrow \mc{C}_{s_i}$ is of the form $\mc{A}_{s_i}=(B_i\times B_i^{t})^4$, where $B_i$ is of $\GL_2(F_i)$-type. %
By assumption, $\mc{A}_{s_i}$ is isogenous to $B_i^{k}$, where 
 $B_i$ is of $\GL_2(F_i)$-type and  $k\geq 1$.
    % $\mc{G}\rightarrow \mc{C}_{s_i}$ is an abelian scheme such that: for any finite extension $\mb{F}_{q'}/\mb{F}_q$, write $C':=\mc{C}_{s_i}\times \Spec(\mb{F}_{q'})$; then  there are no non-zero maps of abelian schemes 
  %  \[
   % B_i\times C'\rightarrow \mc{G}\times C'\]
    %over $C'$. 
%\end{itemize}   
Letting $F_i$ act on $B_i^{k}$ diagonally, we have 
\begin{align*}
F_i\otimes\mb{Q}_{\ell} & \xhookrightarrow{} \End(\mb{V}_{\ell})^{G^0_{geom, \ell}}\\
& \xhookrightarrow{} \End(\mb{V}_{\ell}\otimes_{\mb{Q}_{\ell}} E_{\lambda})^{G^0_{geom, \ell}}\\
& \simeq \End(\mb{U}\otimes_{E} E_{\lambda})^{G^0_{E_{\lambda}}}\\
&=\End(\mb{U})^{G^0}\otimes_{E} E_{\lambda}.
\end{align*}

%Indeed, by Grothendieck's specialization of \'etale fundamental groups, the right hand side is isomorphic to $Z(\End(\mb{V}\otimes \mb{Q}_{\ell})^{\pi_1(\mc{C}_{\bar{s}_i})})$, where $\bar{s}_i$ is a geometric point living over $s_i$. Recall that  the family $\mc{A}_{s_i}\rightarrow \mc{C}_{s_i}$ is of the form $\mc{A}_{s_i}=(B\times B^{t})^4$, where $B$ is of $\GL_2(F_i)$-type; we therefore  have $F_i\xhookrightarrow{}  \End(\mb{V}\otimes \mb{Q}_{\ell})^{\pi_1(\mc{C}_{\bar{s}_i})}$. 

Note that we have $\mb{V}_{\ell}\cong W^{\oplus k}$ as $\pi_1(\mc{C}_{\bar{s}_i})$-representations.  To show the image of the above map lies in the    center of $\End(\mb{V}_{\ell})^{G^0_{geom, \ell}}$, it suffices to show that\footnote{note that the Tate conjecture does not apply here since we are considering representations of the geometric fundamental group $\pi_1(\mc{C}_{\bar{s}_i})$} 
$\End(W\otimes \overline{\mb{Q}}_{\ell})^{H}=F_i\otimes \Qlbar $, where $W$ denotes the rational $\ell$-adic Tate module of $B$, and $H$ ranges over all finite index subgroups of $\pi_1(\mc{C}_{\bar{s}_i})$: indeed, this implies that  $\End(\mb{V}_{\ell})^{G^0_{geom, \ell}}\Qlbar$ is precisely $M_k(F_i\otimes \Qlbar)$, the ring of $k \times k$-matrices with entries in $F_i\otimes \Qlbar$, and the embedding of $F_i\otimes \mb{Q}_{\ell}$  into $M_k(F_i\otimes \Qlbar)$ is precisely the diagonal one.

But the equality $\End(W\otimes \overline{\mb{Q}}_{\ell})^{H}=F_i\otimes \Qlbar $ follows immediately from \Cref{lemma: stable-all-etale}, as required.

\iffalse 
We have the decomposition
\[
W\otimes \Qlbar =\bigoplus_{\sigma: F_i\xhookrightarrow{} \Qlbar} V_{\sigma},
\]
where each $V_{\sigma}$ is a two-dimensional $\overline{\mb{Q}}_{\ell}$-representation of $\pi_1(\mc{C}_{s_i})$. Moreover, the $F_i$-action on $V_{\sigma}$, induced by the action on $B$, is given by $\sigma$. It remains to show that there are no non-trivial elements in $\Homo_{\pi_1(\mc{C}_{\bar{s}_i})}(V_{\sigma}, V_{\tau})$ for $\sigma\neq \tau$. 

Suppose we have such a non-zero  element $\Phi \in \Homo_{\pi_1(\mc{C}_{\bar{s}_i})}(V_{\sigma}, V_{\tau})$. Since each $V_{\sigma}$ stays irreducible upon restriction to $\pi_1(\mc{C}_{\bar{s}_i})$\footnote{see for example \cite[Lemma 12]{snowdentsimerman}}, the hom space above is one-dimensional, and we see there exists a character $\chi$ of $\Spec(\mb{F}_q)$ such that $V_{\sigma}\otimes \chi \cong V_{\tau}$. Since  $\det V_{\sigma} \cong \Qlbar(-1)$ for each $\sigma$, we have that $V_{\sigma}|_{\mc{C}'_{s_i}}\cong V_{\tau}|_{\mc{C}_{s_i}'}$, where $\mc{C}_{s_i}' \defeq \mc{C}_{s_i}\otimes_{\Spec(\mb{F}_q)} \Spec(\mb{F}_{q^2})$. This contradicts point (1) of \Cref{lemma: trace-field-stable}.
\fi 
%Moreover, by Lemma~\ref{lemma:tracefield} below, passing to $\mc{C}_{s_i}'$ does not change the trace field of Frobenii. Therefore, for all closed points $x\in \mc{C}_{s_i}'$, 
%\[
%\sigma(\tr(\Frob_x))=\tau(\tr(\Frob_x));
%\]
%since $\{\tr(\Frob_x)\}$ generate $F_i$ as $x$ ranges over all points of $\mc{C}'_{s_i}$, we deduce that $\sigma =\tau$, as required.
\end{proof}
\end{proof}

We now assemble all the ingredients to deduce our main result.
\begin{proof}[Proof of Theorem~\ref{thm:main}]
 As in the statement of the theorem, we fix a pair $(g,n)$. Suppose, for the sake of contradiction,  that we have infinitely many distinct trace fields $F_i$'s in $\mf{F}_{g,n}= \bigcup_{\bar{C}, Z, q} \mf{F}(\mscr{L}(\bar{C}, Z))$, corresponding to $\Qlbar$-local systems $\mb{L}_i$ on pointed curves $(\bar{C}_i, Z_i)$ of type $(g,n)$, so that, by \Cref{thm:drinf}, $\mb{L}_i$ comes from an abelian scheme $B_i\rightarrow C_i$ of relative dimension $d$, which is  moreover of $\GL_2(F_i)$-type. We first treat the case when  each  $\mb{L}_i$ has unipotent monodromy around each cusp: in this case, by  Lemma~\ref{lemma:family}, there exists $\mc{A}\rightarrow \mc{C}\rightarrow S$ as in the statement of \Cref{lemma:center},  at which point the lemma allows us to conclude in this case. 
 
We now handle the general case. By the same argument as in \cite[p. 879]{krishnapal}, by passing to the \'etale cover $\varphi_i: C'_i\rightarrow C_i$  trivializing  the $\ell$-torsion of $A_i$, we have that the pullbacks $\varphi_i^*A_i$ has unipotent monodromy or good reduction around each of the cusps of $C'_i$. By the Riemann--Hurwitz formula and the pigeonhole principle we may assume that the  $C'_i$ have the same topological type. By \Cref{lemma: stable-all-etale}, passing from $C_i$ to $C_i'$ does not change the trace field, and hence we are done by the unipotent case.

\end{proof}

\subsection{Proof sketch of Theorem~\ref{thm:diffp}}
\begin{comment}
\begin{prop}
Suppose $X/\mb{C}$ is a smooth curve, and $X\subset \bar{X}$ is a smooth compactification with $\bar{X}\setminus X$ non-empty. Suppose $\pi: A\rightarrow C$ an abelian scheme such that $R^1\pi_*\Qlbar$ decomposes as 
\[
R^1\pi_*\Qlbar=\bigoplus_{i=1}^d \mb{V}_i^{k}
\]
where  
\begin{itemize}
    \item each $\mb{V}_i$ is of rank two, and has unipotent monodromy around one of the cusps of $X$, and 
    \item the $\mb{V}_i$ are pairwise non-isomorphic.
\end{itemize}
Then $A$ is isogenous to $B^k$, where $B$ is of $\GL_2(F)$-type, where $F$ is a number field of degree $d$.
\end{prop}
\end{comment}

Since the proof of \Cref{thm:diffp} is very similar to that of \Cref{thm:main}, we only provide a sketch here. 
\begin{proof}[Proof sketch of Theorem~\ref{thm:diffp}]
The proof is essentially the same as that of Theorem~\ref{thm:main}. Suppose again that there are infinitely many $F_i\in \bigcup_{p, \bar{C}, Z}\mf{F}(\mscr{L}_p(\bar{C}, Z))$, which are of degree $d$ and are furthermore completely split at $p$. Suppose $F_i$ is attached to a local system $\mb{L}_i$, which in turn arises from an  abelian scheme   $B_i\rightarrow C_i$ of $\GL_2(F_i)$-type. By taking the cover $\varphi_i: C'_i\rightarrow C_i$ trivializing the $\ell$-torsion of $A_i$, we have  that $\varphi_i^*\mb{L}_i$ has unipotent monodromy around each cusp; let $\bar{C}'_i$ be a smooth compactification of $C'_i$. By Remark~\ref{rmk:split}, we may assume   that the induced map $\bar{C}_i'\rightarrow \mscr{A}_N^*$ has degree bounded by $g,n,d$, so as in the proof of Theorem~\ref{thm:main} we can find a family of curves  $\mc{C}\rightarrow S$, an abelian scheme $\mc{A}\rightarrow \mc{C}$, and points $s_i$ such that the fiber $\mc{A}_{s_i}$ is isogenous to  $B_i^8$-type. If $S$ is purely in characteristic $p$ for some $p$, then we conclude by Lemma~\ref{lemma:center}; it therefore remains to treat the case when $S$ is in characteristic zero. 

We may therefore pick a complex point $s_{\mb{C}}: \Spec(\mb{C})\rightarrow S$. Let $C_{s_{\mb{C}}}$ be the fiber  of $\mc{C}$  over $s_{\mb{C}}$, and $\pi_{\mb{C}}\mc{A}_{s_{\mb{C}}}\rightarrow C_{s_{\mb{C}}}$ the abelian scheme gotten by basechanging $\mc{A}\rightarrow S$. 

By specialization of prime-to-$p$ fundamental groups, we deduce that $\mb{V}_{s_{\mb{C}}}:= R^1\pi_{\mb{C}}\qbar$ decomposes into rank two pieces, and has infinite local monodromy at the cusps. The infinite local monodromy condition implies that $\mc{A}_{s_{\mb{C}}}$ is in fact isogenous to $B^8$  where $B\rightarrow C_{s_{\mb{C}}}$ is an abelian scheme of $\GL_2(F)$-type, for some number field $F$ of degree $d$. 

We may then spread out the elements of $\End(\mc{A}_{s_{\mb{C}}})$ to an open dense subset  $U\subset S$. Since the $s_i$'s are Zariski dense, we may assume  that $s_i\in U$ for all $i$ (by throwing out the $s_i$'s which lie outside $U$).  Since $\End(\mc{A}_{s_i})=M_8(F_i)$, we deduce that there is an injection $M_8(F)\rightarrow M_8(F_i)$, and hence $F=F_i$ for all $i$, which is a contradiction.
\end{proof}

\iffalse extend scalars to $\overline{\mb{Q}}_{\ell}$ and show that $F_i\xhookrightarrow{} Z(\End(\mb{V}\otimes \overline{\mb{Q}}_{\ell})^{\pi_1(\mc{C}_{\bar{s}_i})})$. We have a decomposition
\[
\mb{V}\otimes \overline{\mb{Q}}_{\ell} \simeq \bigoplus_{\sigma} \mb{W}_{\sigma},
\]
where the direct sum is indexed by embeddings $\sigma: F_i\xhookrightarrow{} \overline{\mb{Q}}_{\ell}$, and each $W_{\sigma}$ is a rank two $\overline{\mb{Q}}_{\ell}$-local system. Moreover, the embedding $F_i\xhookrightarrow{} Z(\End(\mb{V}\otimes \overline{\mb{Q}}_{\ell})^{\pi_1(\mc{C}_{\bar{s}_i})})$ is given by $F$ acting on $W_{\sigma}$ via the embedding $\sigma$.
\fi

\printbibliography[]

@article{tianxiao,
  title={On Goren--Oort stratification for quaternionic Shimura varieties},
  author={Tian, Yichao and Xiao, Liang},
  journal={Compositio Mathematica},
  volume={152},
  number={10},
  pages={2134--2220},
  year={2016},
  publisher={London Mathematical Society}
}

@incollection{maxim,
  title={Notes on motives in finite characteristic},
  author={Kontsevich, Maxim},
  booktitle={Algebra, arithmetic, and geometry},
  pages={213--247},
  year={2009},
  publisher={Springer}
}

@article{litt,
  title={Arithmetic representations of fundamental groups, II: Finiteness},
  author={Litt, Daniel},
  journal={Duke Mathematical Journal},
  volume={170},
  number={8},
  pages={1851--1897},
  year={2021}
}

@article{zuoconstruct,
  title={Constructing abelian varieties from rank 2 Galois representations},
  author={Krishnamoorthy, Raju and Yang, Jinbang and Zuo, Kang},
  journal={arXiv preprint arXiv:2208.01999},
  year={2022}
}

@article{faltings,
  title={Arakelov's theorem for Abelian varieties},
  author={Faltings, Gerd},
  journal={Inventiones mathematicae},
  volume={73},
  number={3},
  pages={337--347},
  year={1983},
  publisher={Springer}
}

@article{xia,
  title={On the deformation of a Barsotti-Tate group over a curve},
  author={Xia, Jie},
  journal={arXiv preprint arXiv:1303.2954},
  year={2013}
}

@article{snowdentsimerman,
  title={Constructing elliptic curves from galois representations},
  author={Snowden, Andrew and Tsimerman, Jacob},
  journal={Compositio Mathematica},
  volume={154},
  number={10},
  pages={2045--2054},
  year={2018},
  publisher={London Mathematical Society}
}

@incollection{delignefinitude,
  title={Un th{\'e}oreme de finitude pour la monodromie},
  author={Deligne, Pierre},
  booktitle={Discrete groups in geometry and analysis},
  pages={1--19},
  year={1987},
  publisher={Springer}
}

@article{katotrihan,
  title={On the conjectures of Birch and Swinnerton-Dyer in characteristic p> 0},
  author={Kato, Kazuya and Trihan, Fabien},
  journal={Inventiones mathematicae},
  volume={153},
  number={3},
  pages={537--592},
  year={2003},
  publisher={Springer-Verlag}
}

@article{krishnapal,
  title={Rank 2 local systems and abelian varieties II},
  author={Krishnamoorthy, Raju and P{\'a}l, Ambrus},
  journal={Compositio Mathematica},
  volume={158},
  number={4},
  pages={868--892},
  year={2022},
  publisher={London Mathematical Society}
}

@article{jostyau,
  title={Harmonic mappings and algebraic varieties over function fields},
  author={Jost, J{\"u}rgen and Yau, Shing-Tung},
  journal={American Journal of Mathematics},
  volume={115},
  number={6},
  pages={1197--1227},
  year={1993},
  publisher={JSTOR}
}

@article{chin,
  title={Independence of $\ell$ of monodromy groups},
  author={Chin, CheeWhye},
  journal={Journal of the American Mathematical Society},
  volume={17},
  number={3},
  pages={723--747},
  year={2004}
}

@misc{memaeda,
  doi = {10.48550/ARXIV.2211.06120},
  
  url = {https://arxiv.org/abs/2211.06120},
  
  author = {Lam, Yeuk Hay Joshua},
  
  title = {Motivic local systems on curves and Maeda's conjecture},
  
  publisher = {arXiv},
  
  year = {2022},
  
  copyright = {Creative Commons Attribution 4.0 International}
}

@book{neukirch2013algebraic,
  title={Algebraic number theory},
  author={Neukirch, J{\"u}rgen},
  volume={322},
  year={2013},
  publisher={Springer Science \& Business Media}
}

@article{sga1960grothendieck,
  title={Grothendieck, Rev{\^e}tements {\'e}tales et groupe fondamental},
  author={SGA, A},
  journal={S{\'e}minaire de g{\'e}om{\'e}trie alg{\'e}brique du Bois Marie},
  volume={1961},
  year={1960}
}

@article{otabe2018tame,
  title={The tame fundamental group schemes of curves in positive characteristic},
  author={Otabe, Shusuke},
  journal={arXiv preprint arXiv:1802.01111},
  year={2018}
}

@article{drinfeld2010conjecture,
  title={On a conjecture of Deligne},
  author={Drinfeld, Vladimir},
  journal={arXiv preprint arXiv:1007.4004},
  year={2010}
}
%\printbibliography[
%heading=bibintoc,
%title={References}
%]  % Bibliography database file Y.bib

\end{document}